\documentclass[12pt]{amsart}

\usepackage{amsmath,
amsthm,amsfonts,amssymb, latexsym,enumerate,url}

\usepackage{microtype}

\usepackage[colorlinks=true, 
linkcolor = blue,
citecolor=black,
urlcolor=black,
]{hyperref}

\usepackage{comment}

\makeatletter
\renewenvironment{proof}[1][\proofname]{\par\pushQED{\qed}%
	\normalfont \topsep6\p@\@plus6\p@\relax
	\trivlist
	\item\relax
		{\bfseries
	#1\@addpunct{.}}\hspace\labelsep\ignorespaces}{%
	\popQED\endtrivlist\@endpefalse
}
\makeatother

\renewcommand{\bf}{\textbf}

\newcommand{\ce}[1]{\mathcal{#1}}

\renewcommand{\bar}[1]{\mkern 1.5mu\overline{\mkern-1.5mu#1\mkern-1.5mu}\mkern 1.5mu}

\newcommand{\ls}{<}
\newcommand{\g}{>}

\newcommand{\OR}{\mathcal{O}}

\newcommand{\sm}{\setminus}

\newcommand{\nr}[1][G]{\textnormal{\bf{N}}_{#1}}
\newcommand{\cn}[1][G]{\textnormal{\bf{C}}_{#1}}
\newcommand{\op}{\textnormal{\bf{O}}}
\newcommand{\zn}{\textnormal{\bf{Z}}}

\newcommand{\IRR}{\textnormal{Irr}}

\newcommand{\AU}{\textnormal{Aut}}

\newcommand{\syl}[1][p]{\textnormal{Syl}_{#1}}
\newcommand{\core}[1][G]{\textnormal{core}_{#1}}

\renewcommand{\subset}{\subseteq}
\renewcommand{\supset}{\supseteq}

\begin{document}

\theoremstyle{plain}

\newtheorem{thm}{Theorem}[section]
\newtheorem{lem}[thm]{Lemma}
\newtheorem{conj}[thm]{Conjecture}
\newtheorem{pro}[thm]{Proposition}
\newtheorem{cor}[thm]{Corollary}
\newtheorem{que}[thm]{Question}
\newtheorem{rem}[thm]{Remark}
\newtheorem{defi}[thm]{Definition}
\newtheorem{hyp}[thm]{Hypothesis}

\newtheorem*{thmA}{THEOREM A}
\newtheorem*{thmB}{THEOREM B}
\newtheorem*{corB}{Corollary B}

\newtheorem*{thmC}{THEOREM C}
\newtheorem*{conjA}{CONJECTURE A}
\newtheorem*{conjB}{CONJECTURE B}
\newtheorem*{conjC}{CONJECTURE C}

\newtheorem*{thmAcl}{Main Theorem$^{*}$}
\newtheorem*{thmBcl}{Theorem B$^{*}$}

\newtheorem{theo}{Theorem}
\renewcommand{\thetheo}{\Alph{theo}}

\numberwithin{equation}{section}

\marginparsep-0.5cm

\renewcommand{\thefootnote}{\fnsymbol{footnote}}
\footnotesep6.5pt

\title{On The Eaton--Moret\'o Conjecture for Principal Blocks of Finite Groups}

\author[Arranz]{Asier Arranz}

\address[Arranz]{Departamento de Matemáticas Fundamentales, UNED, 28040, Madrid, Spain}
\email{aarranz97@alumno.uned.es}

\author[G\'omez--Serrano]{Javier G\'omez--Serrano}
\address[G\'omez--Serrano]{Department of Mathematics, 
Brown University, 
Providence, RI 02912, USA}
\email{javier\_gomez\_serrano@brown.edu}

\author[Navarro]{Gabriel Navarro}
\address[Navarro]{Departament de Matem\`atiques, Universitat de Val\`encia, 46100 Burjassot,
Val\`encia, Spain}
\email{gabriel@uv.es}

\author[Schaeffer Fry]{A. A. Schaeffer Fry}
\address[Schaeffer Fry]{Department of Mathematics, University of Denver, Denver, CO 80210, USA}
\email{mandi.schaefferfry@du.edu}



\begin{abstract}
Let $G$ be a finite group and let $p$ be a prime. If $P$ is a nonabelian Sylow
$p$-subgroup of $G$ and $m(P)$ is the smallest non-linear
irreducible character degree of $P$, we prove that there exists
$\chi \in {\rm Irr}(G)$ in the principal $p$-block of $G$ 
such that $1<\chi(1)_p\le m(P)$, giving one inequality of the Eaton--Moret\'o conjecture for principal blocks. This, assuming Dade's Projective conjecture,
implies  the  Eaton--Moret\'o conjecture for principal blocks.
\end{abstract}

\thanks{JGS has been partially supported by the U.S. National Science Foundation, under Grants DMS-2245017, DMS-2247537 and DMS-2434314, and by a Simons Fellowship. AASF is partially supported by a grant from the U.S. National Science Foundation, Award No. DMS-2439897.  The research of GN is supported by Grant PID2022-137612NB-I00
 funded by MCIN/AEI/ 10.13039/501100011033 and ERDF ``A way of making Europe"}

\maketitle

\section{Introduction}

Let $B$ be a Brauer $p$-block of a finite group $G$ with defect group $P$. The recently-proved (see \cite{KM13, MNST,  Ruh25}) Height Zero Conjecture of Brauer asserts that  every irreducible complex character in $B$ has height zero if and only if $P$ is abelian.

Suppose now that $P$ is nonabelian. In \cite{EM}, C. Eaton and A. Moret\'o conjectured that $m=h$, where $p^m$ is the smallest degree greater than $1$ of an irreducible character of $P$, and $h$ is the smallest positive height among the irreducible characters in $B$. Increasing evidence supports this conjecture
(\cite{BM, FLZ, MMR, MS, NEM}).

Eaton and Moret\'o showed in \cite{EM} that Dade's Projective Conjecture implies the inequality $m \leq h$. Since Dade's conjecture is widely believed to hold, attention has therefore focused on proving the reverse inequality $h \leq m$.
In the most important case of principal blocks, where $P$ is a nonabelian Sylow $p$-subgroup of $G$,  the inequality $h \leq m$ amounts to showing that there exists $\chi \in {\rm Irr}(G)$ in the principal $p$-block of $G$ such that
$$
1<\chi(1)_p \leqslant p^m,
$$
where $p^m$ is the smallest non-linear irreducible
character degree of $P$. This is precisely what we prove in this paper.

\begin{theo}\label{thmA:A}
Let $G$ be a finite group, let $p$ be a prime, and let $P$ be a nonabelian Sylow $p$-subgroup of $G$. If the smallest non-linear irreducible character degree of $P$ is $p^m$, then there exists an irreducible complex character $\chi \in \operatorname{Irr}(G)$ in the principal $p$-block of $G$ such that
\[
1 < \chi(1)_p \leqslant p^m.
\]
\end{theo}

The proof of Theorem \ref{thmA:A} uses the Classification of Finite Simple Groups in several ways, and also depends on the recently-proved strong forms of the McKay conjecture (\cite{CS25}, \cite{R23}).

\begin{corB}
Let $G$ be a finite group. If Dade's Projective conjecture holds for the principal block of $G$,
then so does the Eaton--Moret\'o conjecture.
\end{corB}

\medskip

\bigskip

\noindent
\bf{Acknowledgement.}~~This paper grew out of a conversation between Navarro and G\'omez--Serrano at the ICM in Philadelphia. Still in awe of the recent advances in artificial intelligence in mathematics, Navarro remarked that he felt that the representation theory of finite groups was, in some sense, still {\sl safe} from AI.
G\'omez--Serrano replied: ``Give me a very good theorem that you would like to prove, not the deepest conjecture in your area, but a genuinely interesting and difficult one.'' The theorem Navarro proposed was a weaker 
block--free version of Theorem \ref{thmA:A} of this paper. After a few days and several suggestions, AI produced a proof of this result. For this purpose, we used a workflow involving ChatGPT-5.6-Sol and Claude Fable 5, with G\'omez--Serrano and Navarro setting up the problem and asking Navarro for further input and advice approximately every $\sim 24$h.
A few days later, again with some help from AI, we managed to prove Theorem \ref{thmA:A} of this paper. The proof presented here is the fruit of many simplifications and revisions of the original one, and has been written entirely by us.

\section{Preliminaries}

We begin with a few elementary observations. Our notation for characters follows \cite{CTFG} and \cite{CTN}. As usual, given a finite group $G$ and $N\lhd G$, we identify $\IRR(G/N)$ with the set of irreducible characters of $G$ whose kernel contains $N$.   If $p$ is a prime and $G$ is a finite group, we use the standard notation $\IRR_{p'}(G)$ to denote the irreducible
characters of $G$ of degree not divisible by $p$.

\begin{lem}\label{lem:ind_p_bound}
Let $H\subset G$ be finite groups, $p$ a prime, and let $\psi\in \IRR(H)$. Then there exists $\chi\in \IRR(G|\psi)$ such that $\chi(1)_p\leqslant |G:H|_p\psi(1)_p$. 
\end{lem}
 
\begin{proof}
This is clear since 
\[|G:H|_p \psi(1)_p = (\psi^G(1))_p = \left( \sum_{\chi\in \IRR(G|\psi)} [\chi,\psi^G] \chi(1)\right)_p \geqslant \min_{\chi\in \IRR(G|\psi)} \chi(1)_p\, .\]
\end{proof}

\begin{lem}\label{lem:perf_ext}
Let $N\lhd G$ be finite groups, let $p$ be a prime, and let $\theta\in \IRR_{p'}(N)$ be $G$-invariant. If $N$ is perfect and $G/N$ is a $p$-group, then $\theta$ extends to $G$.
\end{lem}
 
\begin{proof}
Since $N/\ker(\det(\theta))$ is abelian, we have $N=\ker(\det(\theta))$ and thus $o(\theta)=|N:\ker(\det(\theta))| = 1$. Therefore, $(o(\theta)\theta(1), |G:N|)=1$ and the result follows from Corollary 6.28 of \cite{CTFG}. 
\end{proof}

We will also use a consequence of the proof of the inductive condition of the McKay conjecture, now a theorem.
The notation is explained in \cite{R23}.

\begin{thm}\label{thm:mckay}
Let $K\lhd G$ be finite groups, let $p$ be a prime, and let  $P$ be a Sylow $p$-subgroup of $K$. Then there exists a bijection
\[{}^*:\IRR_{p'}(K)\rightarrow \IRR_{p'}(\nr[K](P))\]
such that the character triples $(G_\theta,K,\theta)$ and $(\nr(P)_{\theta^*},\nr[K](P),\theta^*)$ are isomorphic for every $\theta\in \IRR_{p'}(K)$. 
\end{thm}

\begin{proof}
This follows from \cite[Theorem B]{R23} (and the remark following \cite[Conjecture~A]{R23}), which holds by the proof of the inductive McKay condition (see \cite{CS25}). 
\end{proof}

\begin{cor}\label{cor:mckay_p_part}
Let $K\lhd G$ be finite groups, let $p$ be a prime, and let $P$ be a Sylow $p$-subgroup of $K$. If $\xi\in \IRR(\nr(P)/P)$, then there exists $\chi\in \IRR(G)$ with $\chi(1)_p = \xi(1)_p$. 
\end{cor}

\begin{proof}
We view $\xi$ as a character of $\nr(P)$ with $P\subset \ker(\xi)$. Let $\varphi\in \IRR(\nr[K](P))$ lie under $\xi$. Since $P\subset \ker(\xi_{\nr[K](P)}) \subset \ker(\varphi)$ and $\nr[K](P)/P$ is a $p'$-group, $\varphi$ is of $p'$-degree. By Theorem~\ref{thm:mckay}, there exist $\theta\in \IRR_{p'}(K)$ and an isomorphism of character triples ${}^*:(G_\theta,K,\theta)\rightarrow (\nr(P)_\varphi,\nr[K](P),\varphi)$. In particular, $|G_\theta:K| = |\nr(P)_\varphi:\nr[K](P)|$. Since $G=K\nr(P)$ by the Frattini argument, we deduce that $|G:G_\theta| = |\nr(P):\nr(P)_\varphi|$.

Let $\eta\in \IRR(\nr(P)_\varphi|\varphi)$ be the Clifford correspondent of $\xi$, and let $\psi\in \IRR(G_\theta|\theta)$ with $\psi^*=\eta$. By \cite[Lemma 11.24]{CTFG}, we have $\psi(1)/\theta(1)=\eta(1)/\varphi(1)$ and thus $\psi(1)_p = \eta(1)_p$. Since $\chi=\psi^G\in \IRR(G)$ by the Clifford correspondence, we obtain
\[\chi(1)_p = |G:G_\theta|_p \psi(1)_p = |\nr(P):\nr(P)_\varphi|_p \eta(1)_p =(\eta^{\nr(P)}(1))_p = \xi(1)_p, \]
as required. 
\end{proof}

\section{A key result}\label{sec:keyresult}

The key to proving Theorem~\ref{thmA:A} is Theorem~\ref{thm:key} below. In this section, we study the structure of a counterexample of minimal order to this result, with its proof deferred to the next section.

\begin{thm}\label{thm:key}
Let $G$ be a finite group, let $p$ be a prime, and let $P$ be a Sylow $p$-subgroup of $G$. Suppose $Q\lhd P$ is such that $P/Q$ is abelian and nonnormal in $\nr(Q)/Q$. Then there exists $\chi\in \IRR(G)$ such that 
\[1\ls \chi(1)_p \leqslant |P:Q|\, .\]
\end{thm}
 
The following lemma allows us to replace $P/Q$ with a more suitable abelian section.

\begin{lem}\label{lem:section}
Assume the hypotheses and notation of Theorem~\ref{thm:key}. Then there exists $R\in \syl(G)$ such that $Q\subset D=P\cap R\ls P$ and $D\lhd P,R$. Also, $P/D$ is abelian and nonnormal in $\nr(D)/D$.
\end{lem}
 
\begin{proof}
Since $P$ is not normal in  $\nr(Q)$, we may choose a Sylow $p$-subgroup $R$ of $\nr(Q)$ distinct from $P$. Then $Q\lhd R$, and since $P$ and $R$ are $\nr(Q)$-conjugate, the quotient $R/Q$ is abelian. Letting $D= P\cap R$, we obtain that $Q\subset D\ls P$ and $D\lhd P,R$. Since $P/D$ and $R/D$ are distinct Sylow $p$-subgroups of $\nr(D)/D$ the proof is complete.
\end{proof}

Before establishing the following lemma, we introduce some standard notation. Given a group $G$ and a normal subgroup $N\lhd G$, we write $\bar{G} = G/N$ and denote the natural projection by $\bar{\phantom{G}}:G\rightarrow \bar{G}:g\mapsto \bar{g}$. Thus, the image of a subgroup $H\subset G$ is $\bar{H}=HN/N$.

\begin{lem}\label{lem:count_key}
Let $G$ be a counterexample of minimal order to Theorem~\ref{thm:key}. Assume the notation of Lemma~\ref{lem:section}. The following then hold. 
\begin{enumerate}[(a)]
\item $\op_p(G)=1$ and $G=\op^{p'}(G)$. 
\item For every nontrivial $N\lhd G$, we have $PN=RN$, and there exists $n\in N\cap \nr(D)$ with $P^n=R$. 
\item $G$ has a unique minimal normal subgroup $K$. 
\end{enumerate}
\end{lem}

\begin{proof}
Since $\op_p(G)$ is the intersection of all Sylow $p$-subgroups of $G$, we have $\op_p(G)\subset D$. If $|G:\op_p(G)|\ls |G|$, then by the minimality of $|G|$ there exists $\chi\in \IRR(G/\op_p(G))$ such that 
\[1\ls \chi(1)_p\leqslant |(P/\op_p(G)):(D/\op_p(G))| = |P:D|\leqslant |P:Q|\,. \]
This contradicts the choice of $G$ as a counterexample. Therefore, $\op_p(G)=1$.

Now suppose that $V=\op^{p'}(G)\ls G$, and observe that $V$ contains every $p$-subgroup of $G$. In particular, $P$ and $R$ are distinct Sylow $p$-subgroups of $\nr[V](D)$, and thus $P/D$ is abelian and nonnormal in $\nr[V](D)/D$. By the minimality of $|G|\g |V|$, there exists $\varphi\in \IRR(V)$  with
\[1\ls \varphi(1)_p \leqslant |P:D|\leqslant |P:Q|\, .\]
By Corollary 11.29 of \cite{CTFG}, any $\chi\in \IRR(G|\varphi)$ satisfies $\chi(1)_p =\varphi(1)_p$, against the choice of $G$ as a counterexample. Therefore, $G=\op^{p'}(G)$.

Let $1\ls N\lhd G$, write $\bar{G}=G/N$, and assume that $T/N=\bar{P}\cap \bar{R}\ls \bar{P}$. Then $\bar{D}\subset \bar{T}$,  $\bar{D} \lhd \bar{P}$, $\bar{D} \lhd \bar{R}$, and since $\bar{P}/\bar{D}$ and $\bar{R}/\bar{D}$ are abelian, we have $\bar{T}\lhd \bar{P},\bar{R}$. It follows that $\bar{P}/\bar{T}$ is abelian and nonnormal in $\nr[\bar{G}](\bar{T})/\bar{T}$. By the minimality of $|G|\g |\bar{G}|$, there exists $\chi\in \IRR(\bar{G})$
\[1\ls \chi(1)_p\leqslant |\bar{P}:\bar{T}|\leqslant |\bar{P}:\bar{D}|\leqslant |P:D|\leqslant |P:Q|\, .\]
This contradicts the choice of $G$, and hence $PN=RN$. It follows that $P$ and $R$ are Sylow $p$-subgroups of $\nr[PN](D)$, and therefore there is  $xn\in \nr[PN](D)$ with $x\in P$ and $n\in N$ such that $R=P^{xn} =P^n$. This establishes (b).

Let $K_1$ and $K_2$ be distinct minimal normal subgroups of $G$, so that $K_1\cap K_2 = 1$ and $[K_1,K_2]=1$. By part (b), there exists $k_i\in K_i$ such that $P^{k_i} = R$. Then $x=k_1k_2^{-1}$ normalizes $P$. Since $P$ is a Sylow $p$-subgroup of $G$ and $K_1, K_2\lhd G$, recall that
\[P\cap (K_1K_2)=(P\cap K_1)(P\cap K_2) = (P\cap K_1)\times (P\cap K_2)\, .\]
In particular, if $z \in P$ and $z=f_1f_2$ for some $f_i \in K_i$, it follows that $f_i \in P\cap K_i$.
Now, let $y\in P$ and observe that
\[ [y,x] =[y,k_1k_2^{-1}]=[y,k_2^{-1}][y,k_1]^{k_2^{-1}} =[y,k_2^{-1}] [y,k_1] \in  K_2\times K_1\, .\]
Since $[y,x]=y^{-1} y^x\in P$, we deduce that $[y,k_1]\in P\cap K_1$ and $[y,k_2^{-1}] \in P\cap K_2$
for all $y \in P$. In particular, $k_1$ normalizes $P$. Thus $P=P^{k_1} = R$, a contradiction. 
\end{proof}

Next, we study the unique minimal normal subgroup of $G$.

%%%%%%%%%%%%%%%%%%%%%%%%%%%%%%%%%%%%%%%%%%%%%%%%%%%%%%%%%%%%%%%%%%%%%%%%%%%%%%%%%%%%%%%%%%%%%%%%%%%%%%%%%%%%%%%%%%%%%%%%%%%%%%%%%%%%%%%%%%%%%%%%%%%%%%

\begin{lem}\label{lem:ruling_out}
Let $G$ be a counterexample of minimal order to  Theorem~\ref{thm:key}, and assume the notation of Lemmas~\ref{lem:section} and~\ref{lem:count_key}. Then there does not exist a character $\theta\in \IRR_{p'}(K)$ such that the stabilizer $P_\theta$ is a proper normal subgroup of $P$, the quotient $P/P_\theta$ is abelian, and $|P:P_\theta|\leqslant |P:D|$. 
\end{lem}
 
\begin{proof}
Suppose there exists such a character $\theta\in \IRR_{p'}(K)$. Write $T=G_\theta$, and let $S$ be a Sylow $p$-subgroup of $T$ such that $P_\theta = P\cap T\subset S$. We first claim that $\theta$ extends to some $\hat{\theta}\in \IRR(KP_\theta)$. If $K$ is abelian, then it is an elementary abelian $q$-group for some prime $q$. By Lemma~\ref{lem:count_key}(a), $q\neq p$. In this case, $\theta$ is linear and $K\cap P_\theta=1$. Thus $\theta$ extends to $\hat{\theta}\in \IRR(KP_\theta)$.  If $K$ is nonabelian, then $K'=K$ and $K$ is perfect. By Lemma~\ref{lem:perf_ext}, $\theta$ extends to $\hat{\theta}\in \IRR(KP_\theta)$, and the claim follows.

We next show that $S$ is a Sylow $p$-subgroup of $G$. Since $P_\theta\cap K=(P\cap T)\cap K=P\cap K\in \syl(K)$, the index 
\[|KP_\theta:P_\theta| = |K:K\cap P_\theta| = |K:K\cap P|\]
is not divisible by $p$, and thus $|KP_\theta|_p = |P_\theta|$. Applying Lemma~\ref{lem:ind_p_bound} to $\hat{\theta}$, we obtain $\psi\in \IRR(T|\hat{\theta})\subset \IRR(T|\theta)$ such that
\[\psi(1)_p \leqslant |T:KP_\theta|_p \hat{\theta}(1)_p = \frac{|T|_p}{|KP_\theta|_p} = \frac{|S|}{|P_\theta|}\, .\]
By the Clifford correspondence, $\chi = \psi^G$ is an irreducible character of $G$ satisfying
\[\chi(1)_p = |G:T|_p \psi(1)_p\leqslant \frac{|P|}{|S|}\frac{|S|}{|P_\theta|} =|P:P_\theta|\leqslant |P:D|\, .\]
Since $G$ is a counterexample to Theorem~\ref{thm:key}, we must have $|G:T|_p =\psi(1)_p=1$. Thus $S$ is a Sylow $p$-subgroup of $G$.

Write $\bar{G}=G/K$ and use the bar convention. We work to show that $\bar{P}/\bar{P_\theta}$ is nonnormal in $\nr[\bar{G}](\bar{P_\theta})/\bar{P_\theta}$. If $S=P_\theta$, then $P_\theta$ would be a Sylow $p$-subgroup of $G$, and thus $P=P_\theta$ and this is not possible.  We have then $U=\nr[S](P_\theta)\g P_\theta$. To prove that $\bar{P}/\bar{P_\theta}$ is a nonnormal Sylow $p$-subgroup of $\nr[\bar{G}](\bar{P_\theta})/\bar{P_\theta}$, it suffices to show that $\bar{P}$ does not contain every $p$-subgroup of $\nr[\bar{G}](\bar{P_\theta})$. It then suffices to show that $\bar{U}\not\subset \bar{P}$. Assume for a contradiction that $\bar{U}\subset \bar{P}$, or equivalently, that $UK\subset PK$. Let $V=P\cap KU$. Since 
\[|KU:V| = |KU:KU\cap P| = |KUP:P| = |KP:P|\]
is a $p'$-number, $V$ is a Sylow $p$-subgroup of $KU$. We claim that $U$ is also a Sylow $p$-subgroup of $KU$. Observe that both $P\cap K$ and $S\cap K$ are Sylow $p$-subgroups of $K$. Since $P\cap T=P_\theta\subset S$, we have $P\cap K=(P\cap T)\cap K\subset S\cap K$, and thus
\[P\cap K=P_\theta\cap K=S\cap K\, .\]
It follows that
\[S\cap K=P_\theta \cap K\subset U\cap K\subset S\cap K,\]
and equality holds throughout. Hence 
\[|KU:U| = |K:K\cap U| = |K:K\cap S|\]
is not divisible by $p$, and $U$ is a Sylow $p$-subgroup of $KU$, as claimed. Thus there exists $k\in K$ with $V=U^k$. Since both $U=\nr[S](P_\theta)\subset T=G_\theta$ and $k\in K$ stabilize $\theta\in \IRR(K)$, we deduce that $V\subset T\cap P\subset P_\theta$. This is a contradiction since $|V|=|U|=|\nr[S](P_\theta)|\g |P_\theta|$.

We conclude that $\bar{P}/\bar{P_\theta}$ is an abelian, nonnormal Sylow $p$-subgroup of $\nr[\bar{G}](\bar{P_\theta})/\bar{P_\theta}$. By the minimality of $|G|\g |\bar{G}|$, there exists $\chi\in \IRR(\bar{G})$ with 
\[1\ls \chi(1)_p \leqslant |\bar{P}:\bar{P_\theta}| \leqslant |P:P_\theta|\leqslant |P:D|,\]
which contradicts the choice of $G$. 
\end{proof}

\begin{lem}\label{lem:p'_min_nor}
Let $G$ be a counterexample of minimal order to Theorem~\ref{thm:key}. Assume the notation of Lemmas~\ref{lem:section} and~\ref{lem:count_key}. If $K$ is the unique minimal normal subgroup of $G$, then $p$ divides $|K|$. In particular, $K$ is nonabelian. 
\end{lem}

\begin{proof}
Suppose that $K$ is a $p'$-group. Since a normal $p'$-subgroup normalizes a $p$-subgroup if and only if it centralizes it, observe that $\nr[K](D)=\cn[K](D)$ and $\nr[K](P)=\cn[K](P)$. By Lemma~\ref{lem:count_key}, let $k\in \nr[K](D)=\cn[K](D)$ with $D=P\cap P^k\ls P$. In particular, $k\notin \cn[K](P)$.

Suppose first that every member of $\IRR_D(K)$ (the set of $D$-invariant $\theta\in \IRR(K)$) is $P$-invariant. By Lemma~2.2 of \cite{N93}, $\cn[K](D) = \cn[K](P)$ contradicting that $k\in \cn[K](D)\sm \cn[K](P)$.

Thus there exists $\theta\in \IRR_D(K)$ that is not $P$-invariant. Then $D\subset P_\theta\ls P$. Since $P/D$ is abelian, we have $P_\theta\lhd P$ and $P/P_\theta$ is abelian. This, however, contradicts Lemma~\ref{lem:ruling_out} and we conclude that $p$ divides $|K|$.

Finally, if $K$ is abelian, then $K$ is an elementary abelian $q$-group for a prime $q$. By Lemma~\ref{lem:count_key}(a), $q\neq p$ and this contradicts the previous paragraph. Thus $K$ is nonabelian. 
\end{proof}

We now fix some notation for the remainder of this section and make a few observations. Let $G$ be a counterexample of minimal order to Theorem~\ref{thm:key}. By Lemma~\ref{lem:count_key}, $G$ has a unique minimal normal subgroup $K$ and, by Lemma~\ref{lem:p'_min_nor}, $K=S_1\times \cdots \times S_t$, where the $S_i$ are isomorphic nonabelian simple groups transitively permuted by $G$.

Write $P_i = P\cap S_i$ and $D_i = D\cap S_i$ for $1\leqslant i\leqslant t$. Since $P\in \syl(G)$ and $S_i\lhd K\lhd G$, observe that $P\cap K = P_1\times \cdots\times P_t$. By Lemma~\ref{lem:count_key}, we may choose $k\in K\cap \nr(D)$ such that $P\cap P^k = D$. Then 
\[ P_i \cap P_i^k = (P\cap S_i)\cap (P\cap S_i)^k  = P\cap P^k\cap S_i = D\cap S_i =D_i,\]
and
\[D\cap K = (P\cap P^k)\cap K=(P\cap K)\cap (P\cap K)^k = \prod_{i=1}^t (P_i\cap P_i^k) = \prod_{i=1}^t D_i\, .\]
Furthermore, $D_i=D\cap S_i\lhd P\cap S_i= P_i$, and $P_i \cap D=P_i \cap (D_1 \cdots D_t)=D_i$. Note also that 
$P_i$ normalizes $D$ and
\[ P_i/D_i \cong P_iD/D\subset P/D\]
is abelian, since it is isomorphic to a subgroup of the abelian group $P/D$. Finally, observe that
\[|PK:DK| = |P:(P\cap DK)| = |P:(P\cap K)D| = \frac{|P| |D\cap K|}{|D| |P\cap K|}\]
and
\[ \prod_{i=1}^t |P_i:D_i| = \frac{|P\cap K|}{|D\cap K|},\]
which yields
\[|P:D| =|PK:DK|\prod_{i=1}^t |P_i:D_i|\, .\]

Next, we gather information on the $p$-subgroups $D_i$ and $P_i$.

\begin{lem}\label{lem:active_factor}
Let $G$ be a counterexample of minimal order to Theorem~\ref{thm:key}. Assume the notation of the previous 
paragraphs and lemmas. Then there exists a factor $S_i$ of $K$ such that $D_i\ls P_i$. 
\end{lem}

\begin{proof}
Assume that $D_i = P_i$ for all $1\leqslant i\leqslant t$. Then $P_0=P\cap K = D\cap K$ and, since $p$ divides $|K|$ by Lemma~\ref{lem:p'_min_nor}, $P_0\g 1$. Since the element $k$ lies in $K\cap \nr(D)$, we have $P_0^k = (D\cap K)^k = P_0$. Write $N=\nr(P_0)$. Since $P_0=P\cap K\lhd P$, we have that $P_0 = P_0^k$ is normal in $P^k = R$. 
So $P/D \ne R/D$ are Sylow $p$-subgroups of $\nr[N](D)/D$. Write $\bar{N} = N/P_0$ and use the bar notation. Since $\bar{P}/\bar{D}$ is abelian and $\bar{P}/\bar{D}$ is nonnormal in $\nr[\bar{N}](\bar{D})/\bar{D} = \bar{\nr[N](D)}/\bar{D}$, by the minimality of $|G|\g |\bar{N}|$ there exists $\xi\in \IRR(N/P_0)$ such that
\[1\ls \xi(1)_p\leqslant |\bar{P}:\bar{D}|\leqslant |P:D|\, .\]
Now, by Corollary~\ref{cor:mckay_p_part} there exists $\chi\in \IRR(G)$ with $\chi(1)_p = \xi(1)_p$, and this contradicts the choice of $G$ as a counterexample. 
\end{proof}

If $G=A\times B$ and $\alpha\in \IRR(A)$, we will slightly abuse notation and identify $\alpha$ with its unique extension $\hat{\alpha}$ to $G$ such that $B\subset \ker(\hat{\alpha})$.

\begin{lem}\label{lem:inv_factor}
Let $G$ be a counterexample of minimal order to Theorem~\ref{thm:key} and continue to keep the notation of the previous paragraphs and lemmas. If $D_i\ls P_i$, then $S_i$ is $P$-invariant. 
\end{lem}

\begin{proof}
Since $\{S_i\}_{i=1}^t$ is the set of minimal normal subgroups of $K$, $P$ acts on this set by conjugation. Therefore, $P$ acts on $\ce{A} = \{S_i\,|\, D_i\ls P_i\}$, since $D_i=D\cap K$ and $P_i=P\cap K$ and $D \lhd P$. Suppose $\Delta\subset \ce{A}$ is a nontrivial $P$-orbit. If the derived subgroup $P'$ acts transitively on $\Delta$ and $S_i\in \Delta$, the Frattini argument yields $P=P'\nr[P](S_i)$. Since $P'\subset \Phi(P)$, we obtain
\[P =  \nr[P](S_i) \]
contradicting that $\Delta$ is a nontrivial $P$-orbit. Thus there exists a $P'$-orbit $\Lambda\subsetneq \Delta$ containing $S_i\in \Lambda$. Let $T=P_\Lambda$ be the setwise stabilizer of $\Lambda$ in $P$. Then $P'\subset T\ls P$. In particular, $T\lhd P$ and $P/T$ is abelian.

Now, by \cite[Lemma~2.2]{H17} there exists an $\AU(S_i)$-orbit $\OR\subset \IRR(S_i)$ of non-trivial $p'$-degree characters whose length is prime to $p$. Since $\nr[T](S_i)$ is a $p$-group acting on $\OR$, there exists an $\nr[T](S_i)$-invariant character $\alpha\in \OR$. Let $X$ be a right transversal for $\nr[T](S_i)$ in $T$, so that $\prod_{x\in X} \alpha^x$ is a $T$-invariant irreducible character of $\prod_{S_j\in \Lambda} S_j$. Then
\[\theta = \prod_{x\in X} \alpha^x \prod_{S_j\notin \Lambda} 1_{S_j}\]
is a $T$-invariant irreducible character of $K$ of $p'$-degree. Moreover, any element $y\in P$ fixing $\theta$ must stabilize $\Lambda$. Thus $T=P_\theta$. Since $|P:T|=|P:P_\Lambda|$ is the number of translates of $\Lambda$ in $ \Delta$, we have
\[|P:T|\leqslant |\Delta| \leqslant |\ce{A}|\leqslant p^{|\ce{A}|} \leqslant \prod_{S_i\in \ce{A}} |P_i:D_i|\leqslant |P:D|,\]
contradicting Lemma~\ref{lem:ruling_out}. 
\end{proof}
 
The next result reduces Theorem~\ref{thm:key} to a question on almost simple groups.

\begin{thm}\label{thm:reduction}
Let $G$ be a counterexample of minimal order to Theorem~\ref{thm:key}. Let $S_i$ be a factor of $K$ such that $D_i\ls P_i$. Then $A=\nr(S_i)/\cn(S_i)$ is an almost simple group with socle $S_i\cn(S_i)/\cn(S_i)$. Furthermore, there exist distinct Sylow $p$-subgroups $U$ and $V$ of $A$ such that the following hold.
\begin{enumerate}[(a)]
\item There exists $a\in S_i\cn(S_i)/\cn(S_i)$ such that $V=U^a$.
\item $E=U\cap V$ is normal in both $U$ and $V$. 
\item $|U:E|\leqslant |P:D|$. 
\end{enumerate}
\end{thm}

\begin{proof}
Write $X=\nr(S_i)$ and $C=\cn(S_i)$. By Lemma~\ref{lem:inv_factor}, $P\subset X$. Since $K\subset X$, we also have $R=P^k\subset X$. Write $\bar{X} = X/C = A$, and note that the element $a=\bar{k}$ lies in $\bar{S_i} = S_iC/C$. 

Since $A$ is a subgroup of $\AU(S_i)$ containing the nonabelian simple group $\bar{S_i}$ (which is isomorphic to $S_i$), the group $A$ is almost simple with socle $\bar{S_i}=S_iC/C$. 

Write $ V=\bar{P}$ and $U=\bar{R}$, so that $U$ and $V$ are Sylow $p$-subgroups of $A$ with $U=V^a$. We now show that $U=\bar{P}$ and $V=\bar{P^k}$ are distinct Sylow $p$-subgroups. Suppose for a contradiction that $PC = P^k C$. Observe that
\[PC\cap (S_iC) = (P\cap (S_iC))C = (P\cap S_i)(P\cap C)C = (P\cap S_i)C,\]
where the second equality holds because $P\in \syl(X)$, and $S_iC=S_i\times C\lhd X$. Similarly, $P^k C\cap S_i C=(P^k\cap S_i)C$. Since we are assuming that $PC= P^kC$, it follows that 
\[(P\cap S_i)C = (P^k\cap S_i)C\, .\]
Intersecting both sides of the equation with $S_i$ yields
\[(P\cap S_i)C\cap S_i = (P^k\cap S_i)C \cap S_i\, .\]
Since
\[(P\cap S_i)C\cap S_i = (P\cap S_i)(C\cap S_i) = P\cap S_i,\]
and, similarly, $(P^k\cap S_i)C\cap S_i = P^k\cap S_i$, we conclude that $P_i = P\cap S_i = P^k\cap S_i = P_i^k$. This contradicts the fact that $P_i\cap P_i^k = D_i\ls P_i$. Thus $V\neq U$. 

Since $D\lhd  P,R$, we have that  $E_0=\bar{D}$ is normal in both $\bar{P} = V$ and $\bar{R} = U$, and $V/E_0=\bar{P}/\bar{D}$ is abelian. Since $E=U\cap V\supset E_0$, it follows that $E$ is normal in both $U$ and $V$. Finally
\[|U:E|=|V:E|\leqslant |V:E_0| = |\bar{P}:\bar{D}| \leqslant |P:D|,\]
and this completes the proof. 
\end{proof}

\section{Almost simple groups}

The following lemmas follow essentially from part of the work of \cite{BM,MMR}, in which height-one characters are found in principal blocks of many decorated simple groups. (See also \cite{MS}, where arbitrary blocks are considered.) As usual, given a finite group $X$ and a fixed prime $p$, we let $B_0(X)$ denote the principal $p$-block of $X$.

\begin{lem}\label{lem:almost_simple_p}
Let $S$ be a finite nonabelian simple group, $p$ a prime dividing $|S|$, and $S< A\subset \AU(S)$, with $A=\op^{p'}(A)$. Then there exists $\chi\in \IRR(B_0(A))$ such that $S\not\subset \ker(\chi)$ and $\chi(1)_p = p$. 
\end{lem}
 
\begin{proof}
Observe first that $\op_{p'}(A)=1$. Indeed, since $p$ divides $|S|$, we have that $\op_{p'}(A)\cap S$ is a proper normal subgroup of $S$, and thus $\op_{p'}(A)\cap S = 1$. Therefore, $[\op_{p'}(A),S]\subset S\cap \op_{p'}(A)=1$ and $\op_{p'}(A)\subset \cn[A](S)=1$, as claimed. Now, by \cite[Theorem 2.2]{K95}, the Sylow $p$-subgroups of $A$ are nonabelian since $S<A=\op^{p'}(A)$.

 If $S$ is a group of Lie type, then by \cite[Theorem~3.5]{MMR}, there is a height-one character of $A$ in the principal block, which is further constructed in the proof to lie above a nontrivial character of $S$. If $S$ is sporadic, the Tits group ${}^{2}\operatorname{F}_4(2)'$, or the alternating group $A_6$, then the statement is similarly proven within \cite[Propositions~3.1,~3.2]{MS} (see also the remark after \cite[Proposition~3.2]{MS}). For the remaining alternating groups, our assumptions force $p=2$ and $A$ is a symmetric group, and the statement follows from \cite[Theorem~2.1]{BM}, since $A/S$ is cyclic of size 2, and hence any character of $A$ with $\chi(1)_2=2$ is necessarily nontrivial on $S$. 
\end{proof}

\begin{lem}\label{lem:spor}
Let $S$ be a finite nonabelian simple group  and let $p$ be a prime dividing $|S|$ such that $S$ has nonabelian Sylow $p$-subgroups. Further, assume that either $S$ is a sporadic group, an alternating group, the Tits group ${}^2\operatorname{F}_4(2)'$, or a group of Lie type defined in characteristic distinct from $p$.  Then there exists $\chi\in \IRR(S)$ such that $\chi(1)_p = p$, with the exception of the case $(p,S)=(3,\mathrm{Co}_3)$. 
\end{lem}
 
\begin{proof}
For alternating groups, this is part of \cite[Theorem~2.1]{BM}, by considering the principal block. For $S$ of Lie type in nondefining characteristic, it is similarly part of \cite[Proposition~4.2, Theorem~4.7]{BM}. For the remaining groups, this can be seen from the Atlas or GAP Character Table Library \cite{GAP, CtLib}.
\end{proof}

\begin{thm} \label{thm:almost_simple}
Let $A$ be a finite almost simple group with socle $S$. For a prime $p$ dividing $|S|$, let $U\in \syl(A)$, and let $a\in S$ be such that $U^a\neq U$. Assume that $E:=U\cap U^a$ is such that $E\lhd U$ and $E\lhd U^a$.
%and $U/E$ is abelian. 
Then there exists $\chi\in \IRR(A)$ such that $S\not\subset \ker(\chi)$ and $1\ls \chi(1)_p\leqslant |U:E|$. 
\end{thm}

\begin{proof}
Set $A_0:=\op^{p'}(A)$ and note that  $U$, $U^a$, and $E$ all lie in $A_0$. Further, $A_0$ is almost simple with socle $S$.  A character $\psi\in\IRR(A_0)$ that is nontrivial on $S$ and satisfies $1< \psi(1)_p\leq |U:E|$ lies below some character $\chi\in\IRR(A)$ that remains nontrivial on $S$ and has the same $p$-part, by Clifford theory.  We may
therefore assume from now on that
$
 A=\op^{p'}(A).$

If $S\neq A$, then the statement now follows from Lemma \ref{lem:almost_simple_p}, noting that our assumption $U^a\neq U$ means $|U:E|\geqslant p$. Hence, we further assume that $A=S$ is simple. Note here that any $\chi\in\IRR(S)$ with $\chi(1)_p>1$ will necessarily be nontrivial on $S$.

First, suppose that $U$ is abelian. Then we may let $D$ be maximal with respect to inclusion among all intersections of two distinct Sylow $p$-subgroups containing $E$. Then $D$ is also maximal among intersections of distinct Sylow $p$-subgroups, since any strictly larger intersection would still contain $E$. Then by  \cite[Corollary~6]{Zhang}, there is some $p$-block $B$ of $S$ for which $D$ is a defect group. By the ``if" direction of Brauer's height zero conjecture \cite[Theorem~1.1]{KM13} (recalling that $D$ is abelian), every character in $\IRR(B)$ has height zero. That is, for $\chi\in\IRR(B)$, we have 
\[\chi(1)_p=|S:D|_p=\frac{|U|}{|D|}\leqslant |U:E|.\]
Since $|D|<|U|$, we are done in this case.

We therefore now assume that $U$ is nonabelian. By Lemma \ref{lem:spor}, we may then assume that either $(p,S)=(3,\mathrm{Co}_3)$ or $S$ is a simple group of Lie type defined in characteristic $p$. (In the latter case, we may further assume that $S$ is not isomorphic to one of the groups considered in Lemma \ref{lem:spor}.)

Suppose that $(p,S)=(3, \mathrm{Co}_3)$.  In this case, we see from the GAP Character Table Library \cite{GAP, CtLib} that the smallest nontrivial value of $\chi(1)_3$ for $\chi\in\IRR(S)$ is $9$, so it suffices to show that $|U:E|\neq 3$. Recall that $U\neq U^a$ are distinct Sylow $3$-subgroups of $N_S(E)$. If $|U:E|=3$, then $E$ is a maximal subgroup of $U$.  
However, calculation in GAP \cite{GAP} with the realization of $S=\mathrm{Co}_3$ as \verb+PrimitiveGroup(276,3)+ shows us that $
 U\lhd \mathbf{N}_{S}(D)$
for every maximal subgroup $D$ of $U$. (This can be checked with
\verb+Unique(List(MaximalSubgroups(U),D ->IsNormal(Normalizer(S,D),U)))+
 which returns only \verb+true+.) In particular, $U\lhd \mathbf{N}_{S}(E)$, contradicting that $U\neq U^a$.

Finally, let $S$ be a simple group of Lie type defined in characteristic $p$, and such that $S$ is not isomorphic to a group covered by the previous arguments. Then $S=G/\zn(G)$ with $G=\mathbf{G}^F$  a quasi-simple group where $\mathbf{G}$ is a simple, simply-connected algebraic group in characteristic $p$ and $F$ is a Steinberg endomorphism. If $F$ is a Frobenius endomorphism, let $q$ be the power of $p$ such that $F$ defines $\mathbf{G}^F$ over $\mathbb{F}_q$. Otherwise, $S$ is a Suzuki or Ree group and we set $q$ to be such that $q^2=2^{2m+1}$ or $q^2=3^{2m+1}$.

 Then $U=\mathbf{U}^F\zn(G)/\zn(G)$, where $\mathbf{U}$ is the unipotent radical of an $F$-stable Borel subgroup $\mathbf{B}$ of $\mathbf{G}$ and $\mathbf{B}^F=\nr(\mathbf{U}^F)$ by \cite[Corollary 24.11]{MT11}. Note that since $\zn(G)$ is a $p'$-group, it suffices to identify $U$ with $\mathbf{U}^F$. Let $\mathbf{T}\leq\mathbf{B}$ be the corresponding $F$-stable maximal torus with $\mathbf{B}=\mathbf{U}\rtimes\mathbf{T}$. 
 By the Bruhat decomposition \cite[Theorem 24.1]{MT11}, we may find $n_w\in \nr(\mathbf{T})$ outside of $\mathbf{T}^F$ such that $a$ has a lift in $G$ lying in the double coset $\mathbf{B}^F n_w \mathbf{B}^F$. In particular, $U^{a}=U^{n_wb}$ for some $b\in \mathbf{B}^F$. Then $E=U\cap U^{n_w b}=(U\cap U^{n_w})^b$ and $|U:E|=|U:U\cap U^{n_w}|$, since $b$ normalizes $U$. Hence it suffices to consider the intersection $U\cap U^{n_w}$ instead. Now, as $n_w\in\mathbf{G}^F$, we have $U^{n_w}=(\mathbf{U}^{n_w})^F$ and this is the $F$-fixed points of the unipotent radical of another Borel subgroup $\mathbf{B}^{n_w}$ with the same maximal torus $\mathbf{T}$. 

 Now, let $\Phi^+$ be the set of positive roots with respect to $\mathbf{T}$. If $\mathbf{U}_\alpha$ is a root subgroup corresponding to some $\alpha\in\Phi^+$ such that $\mathbf{U}_\alpha$ meets $\mathbf{U}\cap\mathbf{U}^{n_w}$ nontrivially, then $\mathbf{U}_\alpha$ is completely contained in $\mathbf{U}\cap\mathbf{U}^{n_w}$ since the root subgroups are the minimal $\mathbf{T}$-invariant subgroups of $\mathbf{U}$ by \cite[Proposition~11.5]{MT11}. Since this intersection is $F$-stable, we further have the $F$-orbit of $\mathbf{U}_\alpha$ must be completely contained inside $\mathbf{U}^{n_w}$. In particular, $\Phi^+$ is the disjoint union of $I:=\{\alpha\in\Phi^+\mid \mathbf{U}_\alpha\leqslant \mathbf{U}\cap\mathbf{U}^{n_w}\}$ and $J:=\{\alpha\in\Phi^+\mid \mathbf{U}_\alpha\not\leqslant \mathbf{U}\cap\mathbf{U}^{n_w}\}$, and each of $I$ and $J$ are unions of $F$-orbits on $\Phi^+$. Further, our assumption that $U\cap U^a\neq U$ means $\mathbf{U}\cap\mathbf{U}^{n_w}\neq\mathbf{U}$, so $J$ is nonempty. Now, for each $\alpha\in\Phi^+$, there is associated a positive integral power $q_\alpha>1$ of $p$ as in \cite[Proposition~22.2]{MT11}. As $\mathbf{U}$ and $\mathbf{U}\cap\mathbf{U}^{n_w}$ are $F$-stable and $\mathbf{T}$-stable, \cite[Proposition 23.8]{MT11} then yields that $|U|=\sum_{\alpha\in\Phi^+} q_\alpha$ and $|U\cap U^{n_w}|=\sum_{\alpha\in I} q_\alpha$. In particular, this means $|U:E|=|U:U\cap U^{n_w}|=\sum_{\alpha\in J} q_\alpha$. Since $J$ is a nonempty union of $F$-orbits on $\Phi^+$, the exact same argument as \cite[Corollary~23.9]{MT11} yields that this is some power $q^k\geqslant q$ of $q$. 
 
 On the other hand, \cite[Proposition~3.2]{BM} and its proof yields a character in $\IRR(B_0(G))$, trivial on $\zn(G)$, with $1<\chi(1)_p\leq q$. (Note that the exception of groups of type $\operatorname{A}_1$ in loc. cit. have abelian Sylow $p$-subgroups.) This completes the proof, by deflating $\chi$ to $S$.
\end{proof}

We are finally ready to prove Theorem~\ref{thm:key}.

\begin{proof}[Proof of Theorem~\ref{thm:key}]
We let $G$ be a counterexample of minimal order and we follow the notation developed throughout Section \ref{sec:keyresult}. In particular, we let $S_i$ be a simple factor of the unique minimal normal subgroup $K$ of $G$. By Lemmas \ref{lem:active_factor} and \ref{lem:inv_factor}, we may choose this $S_i$ such that $D_i<P_i$, where $D_i:=D\cap S_i$ with $D$ as in Lemma \ref{lem:section} and $P_i:=P\cap S_i$, and such that $S_i$ is $P$-invariant. 
Let $A=\nr(S_i)/\cn(S_i)$. By Theorems~\ref{thm:reduction} and~\ref{thm:almost_simple}, there exists $\psi\in \IRR(\nr(S_i))$ with $S_i\not\subset \ker(\psi_K)$ and $\cn[K](S_i)\subset \ker(\psi_K)$, and such that
\[1\ls \psi(1)_p\leqslant |U:E|\leqslant |P:D|\leqslant |P:Q|\, .\]
Since $\ker(\psi_K)$ is the intersection of the kernels of the irreducible constituents of $\psi_K$, we may choose an irreducible constituent $\theta\in \IRR(K)$ of $\psi_K$ such that $S_i\not\subset \ker(\theta)$ and $\prod_{j\neq i} S_j\subset \cn[K](S_i)\subset  \ker(\theta)$. Thus $\theta = \alpha \prod_{j\neq i} 1_{S_j}$ for a nonprincipal $\alpha\in\IRR(S_i)$. Thus for every $x\in G\sm \nr(S_i)$ we have $\theta^x \neq \theta$. That is, $G_\theta\subset \nr(S_i)$ and, by the Clifford correspondence, $\chi=\psi^G\in \IRR(G)$. Since $P\subset \nr(S_i)$ by Lemma~\ref{lem:inv_factor}, we conclude that
\[1\ls \psi(1)_p = |G:\nr(S_i)|_p \psi(1)_p = \chi(1)_p\leqslant |P:Q|,\]
contradicting the choice of $G$ as a counterexample. 
\end{proof}

\section{Proof of Theorem~\ref{thmA:A}}

We begin work towards a proof of Theorem~\ref{thmA:A}. 
Our notation for blocks follows \cite{BT}. Recall that if $P$ is a nonabelian finite $p$-group, we denote by $m(P)$ the smallest nonlinear irreducible character degree of $P$. If $P$ is abelian, we set $m(P)=\infty$. Thus, if $A$ and $B$ are $p$-groups, we have $m(A\times B)=\min\{m(A),m(B)\}$.

The following lemma guarantees the existence of abelian $p$-sections under the hypotheses of Theorem~\ref{thm:key}.

\begin{lem}\label{lem:abelian_psection}
Let $G$ be a finite group, $p$ a prime, and $P,R$  distinct Sylow $p$-subgroups of $G$. Write $J=P\cap R$ and assume that $|P:J| \le m(P)$. Then $P'\subset J\lhd P$, and $P/J$ is an abelian nonnormal Sylow $p$-subgroup of $\nr(J)/J$. 
\end{lem}

\begin{proof}
Let $1_P \ne \eta\in \IRR(P)$ be an irreducible constituent of $(1_J)^P$. Since $1_P$ is also an irreducible constituent of $(1_J)^P$, we have that  $\eta(1)\le |P:J|-1< m(P)$, and we deduce that $\eta(1)=1$. That is, every irreducible constituent of $(1_J)^P$ is linear. Thus
\[P'\subset \ker((1_J)^P) = \core[P](J)\subset J\,. \]
In particular, $J\lhd P$ and $P/J$ is abelian. Since $m(P)=m(R)$, replacing $P$ by $R$ we similarly deduce that $J\lhd R$ and $R/J$ is abelian. Thus $P$ and $R$ are distinct Sylow $p$-subgroups of $\nr(J)$ and $P/J$ is nonnormal in $\nr(J)/J$. 
\end{proof}

Next, we study the minimal normal abelian subgroups that arise in the proof of Theorem~\ref{thmA:A}. First, we need the following two lemmas.

\begin{lem} \label{lem:prin_p_bound}
Let $G$ be a finite group, let $p$ be a prime, and let $P$ be a Sylow $p$-subgroup of $G$. If $\tau\in \IRR(P)$, then there exists $\chi\in \IRR(B_0(G))$ lying over $\tau$ such that $\chi(1)_p\leqslant \tau(1)$. 
\end{lem}

\begin{proof}
Write $H=P\cn(P)$. Since the Sylow $p$-subgroup $\zn(P)$ of $\cn(P)$ is central, $\cn(P)$ has a normal $p'$-complement $K$ and thus $H=P\times K$. Since $P$ has a unique $p$-block by \cite[Problem~3.1]{BT}, $\tau\in \IRR(B_0(P))$ and thus the character $\hat{\tau} = \tau\times 1_K\in \IRR(H)$ lies in the  principal block of $H$. By   \cite[Theorem 4.14]{BT}, $B_0(H)^G$ is defined and, by Brauer's Third Main Theorem (\cite[Theorem~6.7]{BT}), $B_0(H)^G = B_0(G)$. Now, by \cite[Corollary 6.4]{BT}, we have $(\hat{\tau}^G(1))_p = ( (\hat{\tau}^G)_{B_0(G)}(1))_p$, where
\[(\hat{\tau}^G)_{B_0(G)}  = \sum_{\psi\in \IRR(B_0(G))} [\hat{\tau}^G,\psi] \psi\, .\]
Thus there exists $\chi\in \IRR(B_0(G))$ lying over $\tau$ such that $\chi(1)_p\leqslant (\hat{\tau}^G(1))_p  = \tau(1)$. 
\end{proof}

The first part of the following lemma is \cite[Lemma~4.2]{MMR}.

\begin{lem} \label{lem:prin_quot}
Let $N\lhd G$, and let $Q\in \syl(N)$. If $\cn(Q)\subset N$, then $B_0(G)$ is the unique block of $G$ covering $B_0(N)$. In particular, 
\[\IRR(G/N)\subset \IRR(B_0(G))\, .\]
Hence, if $Z\lhd G$ is an abelian $p$-subgroup, then $\IRR(G/\cn(Z))\subset \IRR(B_0(G))$. 
\end{lem}
 
\begin{proof}
Since $B_0(G)$ covers $B_0(N)$ by  \cite[Theorem 9.2]{BT}, it suffices to prove that if $B$ is a block of $G$ covering $B_0(N)$, then $B=B_0(G)$. Let $P$ be a defect group for $B$ such that $P\cap N= Q$ (\cite[Theorem~9.26]{BT}). Then $\cn(P)\subset \cn(Q)\subset N$ and, by \cite[Lemma 9.20]{BT}, $B$ is regular with respect to $N$. By \cite[Theorem 9.19]{BT}, $B=B_0(N)^G$. By Brauer's Third Main Theorem, $B=B_0(G)$ and the first part is complete.
For the second part, take $N=\cn(Z)$ and apply the first.
\end{proof}

\begin{thm}\label{thm:remove_pabelian}
Let $G$ be a finite group, let $p$ be a prime, and let $P$ be a nonabelian Sylow $p$-subgroup of $G$. If $Z$ is an abelian normal $p$-subgroup of $G$, then either Theorem~\ref{thmA:A} holds for $G$ or $m(P)=m(P/Z)$. 
\end{thm}

\begin{proof}
Let $\theta\in \IRR(P)$ with $\theta(1)=m(P)$. If $Z\subset \ker(\theta)$, we are done. Otherwise,  let $1_Z\neq \lambda\in \IRR(Z)$ lie under $\theta$. Let $T=G_\lambda$ be the stabilizer of $\lambda$, and let $\eta\in \IRR(P_\lambda|\lambda)$ be the Clifford correspondent of $\theta$. Then
\[m(P)=|P:P_\lambda| \eta(1)\, .\]
Let $R$ be a Sylow $p$-subgroup of $T$ containing $P_\lambda$. By Lemma~\ref{lem:ind_p_bound}, there exists $\xi\in\IRR(R|\eta)$ with 
\[\xi(1)\leqslant |R:P_\lambda|\eta(1)\, .\]
By Lemma~\ref{lem:prin_p_bound}, there exists $\psi\in \IRR(B_0(T))$ lying over $\xi$ such that
\[\psi(1)_p\leqslant |R:P_\lambda|\eta(1)\, .\]
Since $\psi$ lies over $\lambda$, we have $\chi = \psi^G\in \IRR(G)$. By  \cite[Corollary 6.2]{BT} and Brauer's Third Main Theorem, $\chi\in \IRR(B_0(G))$. Moreover, 
\[\chi(1)_p = |G:T|_p\psi(1)_p\leqslant \frac{|P|}{|R|} |R:P_\lambda| \eta(1) = |P:P_\lambda|\eta(1) = m(P)\, .\]
Thus we may assume that $\chi(1)_p = 1$, so that $R$ is a Sylow $p$-subgroup of $G$.

Let $g\in G$ with $P=R^g$, and write 
\[\mu= \lambda^g,\quad J=(P_\lambda)^g,\quad \xi =\eta^g\, .\]
Thus $\mu$ is $P$-invariant, and $\xi^P(1)=m(P)$.

Suppose that $\xi^P\in \IRR(P)$. If $\mu$ extends to $\nu\in \IRR(P)$, we have that $\xi^P \bar{\nu}\in \IRR(P)$ is such that $(\xi^P\bar{\nu})_Z = \xi^P (1)\mu \bar{\mu}  = (\xi^P\bar{\nu})(1)1_Z$ and thus $m(P/Z)=m(P)$. Suppose that $\mu$ does not extend to $P$. By Lemma~\ref{lem:prin_p_bound}, let $\chi\in \IRR(B_0(G))$ lying over $\xi^P$ such that $\chi(1)_p \leqslant \xi^P(1) = m(P)$. If $\chi(1)_p \g 1$, we are done. Otherwise, $\chi$ is of $p'$-degree  and thus its Clifford correspondent $\psi\in \IRR(G_\mu|\mu)$ is of $p'$-degree. Since $\mu$ does not extend to $P$, it follows that every member of $\IRR(P|\mu)$ has degree divisible by $p$ and thus $p$ divides $\psi(1)$, a contradiction. 

If $\xi^P$ is not irreducible, then every constituent of $\xi^P$ is linear. In particular, $\xi=\eta^g$ and $\eta$ are linear. Also, $|P:J| = |P^g:J| = m(P)$. Write $J_0 = P_\lambda$. Since $P_\lambda\subset R\subset G_\lambda$, we have
\[P_\lambda\subset P\cap R
\quad\hbox{ and }\quad
P_\lambda=G_\lambda\cap P\supset R\cap P\, .\]
Thus $P_\lambda =P\cap R$.

Write $C=\cn(Z)$ and $\bar{G}=G/C$. Since $C$ fixes $\lambda$, we have $P\cap C\subset P_\lambda$, and thus
\[|\bar{P}:\bar{P_\lambda}|=|P(P_\lambda C):P_\lambda C| = |P:P\cap P_\lambda C| = |P:P_\lambda(P\cap C)| = |P:P_\lambda| = m(P)\, .\]
We also have $\bar{P}\cap \bar{R} = \bar{P_\lambda}$. Indeed, if $x\in P$ is such that $\bar{x}\in \bar{R}$, we can write $x=yc$ for $y\in R\subset T$ and $c\in C\subset T$. Thus $x\in P\cap T= P_\lambda$. This proves that $\bar{P}\neq \bar{R}$. Lemma~\ref{lem:abelian_psection} applies and, by Theorem~\ref{thm:key}, there exists $\chi\in \IRR(G/C)\subset \IRR(B_0(G))$ such that
\[1\ls \chi(1)_p\leqslant |\bar{P}:\bar{P_\lambda}| = m(P),\]
and this completes the proof of the theorem. 
\end{proof}

%%%%%%%%%%%%%%%%%%%%%%%%%%%%%%%%%%%%%%%%%%%%%%%%%%%%%%%%%%%%%%%%%%%%%%%%%%%%%%%%%%%%%%%%%%%%%%%%%%%%

To handle nonabelian minimal normal subgroups in the proof of Theorem~\ref{thmA:A}, we use the following results.

\begin{lem} \label{lem:prin_perf}
Let $N\lhd G$ be perfect, and let $\theta\in \IRR_{p'}(B_0(N))$. Then there exists $\chi\in \IRR(B_0(G))$ such that $ \chi(1)_p = |G:G_\theta|_p$. 
\end{lem}

\begin{proof}
Let $T=G_\theta$, and let $S$ be a Sylow $p$-subgroup of $T$. We first prove that there exists $\psi\in \IRR(B_0(T)|\theta)$ of $p'$-degree. Since $N$ is perfect and $\theta(1)$ is a $p'$-number, $\theta$ extends to $\hat{\theta}\in \IRR(NS)$ by Lemma~\ref{lem:perf_ext}. Since $B_0(NS)$ covers $B_0(N)$ and there is a unique block of $NS$ covering $B_0(N)$ by \cite[Corollary 9.6]{BT}, it follows that $\hat{\theta}\in B_0(NS)$. Now, $NS\lhd N\nr[T](S)$ and $B_0(N\nr[T](S))$ covers $B_0(NS)$. By  \cite[Theorem 9.4]{BT}, there exists $\eta\in \IRR(B_0(N\nr[T](S)))$ lying over $\hat{\theta}$. Applying  \cite[Corollary 11.29]{CTFG}, we obtain $\eta(1)_p\leqslant |N\nr[T](S):NS|_p \hat{\theta}(1)_p = 1$. By Brauer's Third Main Theorem and \cite[Problem 4.3]{BT}, $B_0(N\nr[T](S))^T =B_0(T)$. Thus, by \cite[Corollary 6.4]{BT}, there exists $\psi\in \IRR(B_0(T))$ lying over $\eta$ such that $\psi(1)_p\leqslant \eta(1)_p = 1$. By the Clifford correspondence, $\chi = \psi^G\in \IRR(G)$. By \cite[Corollary 6.2, Theorem 6.7]{BT}, $\chi$ lies in the principal block of $G$. Since $\chi(1)_p = |G:T|_p \psi(1)_p = |G:T|_p$, we are done. 
\end{proof}

\begin{lem}\label{lem:simple_product}
Let $G$ be a finite group, let $p$ be a prime, and let $P$ be a Sylow $p$-subgroup of $G$. Suppose $K=S_1\times \cdots\times S_t$ is a nonabelian minimal normal subgroup of $G$, where the $S_i$ are isomorphic nonabelian simple groups of order divisible by $p$. If $P$ acts nontrivially on the set $\{S_i\}_{i=1}^t$, then there exists $\chi\in \IRR(B_0(G))$ with $\chi(1)_p = p$. 
\end{lem}

\begin{proof}
Let $N=\bigcap_{i=1}^t \nr(S_i)$ be the kernel of the action of $G$ on $\{S_i\}_{i=1}^t$, and write $\bar{G} = G/N$, and use the bar convention. Let $U=P\cap N\in \syl(N)$ and note that $P\cap K\subset U$. Since $P\cap K=\prod_{i=1}^t (P\cap S_i)$ and $P\cap S_i\in \syl(S_i)$ is nontrivial, any element centralizing $P\cap K$ necessarily acts trivially on $\{S_i\}_{i=1}^t$. Thus
\[\cn(U)\subset \cn(P\cap K)\subset N\]
and, by Lemma~\ref{lem:prin_quot}, $\IRR(\bar{G})\subset \IRR(B_0(G))$. We shall find $\chi\in \IRR(\bar{G})$ with $\chi(1)_p = p$.

Let $\Delta\subset \{S_i\}_{i=1}^t$ be a nontrivial $P$-orbit, and let $S_i\in \Delta$. Since $|P:\nr[P](S_i)|= |\Delta|\g 1$, we may let $\nr[P](S_i)\subset R\ls P$ with $|P:R|=p$. Let $\Lambda\subset \Delta$ be the $R$-orbit of $S_i$ so that $|\Lambda| = |R:\nr[P](S_i)| = |\Delta|/p$. In particular, $\Lambda\subsetneq \Delta$. Moreover, $R=P_\Lambda $ is the setwise stabilizer of $\Lambda$ in $P$. By \cite[Lemma~3.3]{MMR}, there exists a nonprincipal and $\nr[P](S_i)$-invariant $\alpha\in \IRR_{p'}(B_0(S_i))$. Let $X$ be a right transversal for $\nr[P](S_i)$ in $R$, and let 
\[\theta = \prod_{x\in X} \alpha^x \prod_{S_j\notin \Lambda} 1_{S_j}\in \IRR_{p'}(B_0(K))\, .\]
Then $\theta$ is $R$-invariant and its stabilizer in $P$ is contained in $P_\Lambda = R$. That is, $R=P_\theta$.

Let $T=G_\theta$, and suppose first that $R$ is a Sylow $p$-subgroup of $T$. By Lemma~\ref{lem:prin_perf}, there exists $\chi\in \IRR(B_0(G))$ such that $\chi(1)_p = |G:T|_p = |P:R|=p$.

Otherwise, let $Q_0\g R$ be a Sylow $p$-subgroup of $T$. Then ${\bf N}_{Q_0}(R)>R$, and we can take $Q$ contained in $T$ with $|Q:R|=p$. 
Now, $U=P\cap N \le {\bf N}_P(S_i) \le R \le Q$, and therefore
$U=R\cap N=Q\cap N$, since $U$ is a Sylow $p$-subgroup of $N$.
Hence $|\bar Q: \bar R|=p=|\bar P:\bar R|$. 
If $\bar Q \subseteq \bar P$, then $Q \subseteq PN$.  
Let $q=xn \in Q$, where $x \in P$ and $n \in N$.
Since $Q \subseteq T=G_\theta$, we deduce that $x \in P_\Lambda=R$, and then $\bar Q \subseteq \bar R$, a contradiction. We deduce that $\bar P/\bar R$ is not normal in $\nr[\bar{G}](\bar{R})/\bar{R}$.
 By Theorem~\ref{thm:key} there exists $\chi\in \IRR(\bar{G})\subset \IRR(B_0(G))$ with 
$1\ls \chi(1)_p \leqslant |\bar{P}:\bar{R}| = p,$
as required. 
\end{proof}

It remains to deal with the case where a Sylow $p$-subgroup $P$ of $G$ acts trivially on the simple factors of a nonabelian minimal normal subgroup $K$.

\begin{thm}\label{thm:nonabelian_min}
Let $G$ be a finite group, let $p$ be a prime, and let $P$ be a nonabelian Sylow $p$-subgroup of $G$. Suppose $K=S_1\times \cdots\times S_t$ is a nonabelian minimal normal subgroup of $G$, where the $S_i$ are isomorphic nonabelian simple groups of order divisible by $p$. Then either Theorem~\ref{thmA:A} holds for $G$ or $|G:K\times \cn(K)|$ is not divisible by $p$.
\end{thm}

\begin{proof}
If $P$ acts nontrivially on $\{S_i\}_{i=1}^t$, then by Lemma~\ref{lem:simple_product} there exists $\chi\in \IRR(B_0(G))$ with $\chi(1)_p = p$, and the result follows. Thus we may assume that $P$ acts trivially on $\{S_i\}_{i=1}^t$.

Write $X_i = \nr(S_i)$ and $C_i = \cn(S_i)\lhd X_i$, so that $A_i=X_i/C_i$
is almost simple with socle $S_iC_i/C_i$. Suppose that $p$ divides $|A_i:(S_iC_i/C_i)| = |X_i:S_iC_i|$. Then $Y_i/C_i=\op^{p'}(A_i)\g (S_iC_i/C_i)$. By Lemma~\ref{lem:almost_simple_p}, there exists $\eta\in \IRR(B_0(Y_i))$ with 
$\eta(1)_p=p$, and such that $S_i\not\subset \ker(\eta_K)$ and $\prod_{j\neq i} S_j\subset C_i\cap K \subset \ker(\eta_K)$. Since $\ker(\eta_K)$ is the intersection of the kernels of the irreducible constituents of $\eta_K$, we may choose an irreducible constituent $\theta\in \IRR(K)$ such that $S_i\not\subset \ker(\theta)$, and $\prod_{j\neq i} S_j\subset \ker(\theta)$. Thus $\theta = \alpha \prod_{j\neq i} 1_{S_j}$ for a nonprincipal $\alpha\in \IRR(S_i)$, and we deduce that $G_\theta\subset \nr(S_i)=X_i$. Since $B_0(X_i)$ covers $B_0(Y_i)$ (\cite[Theorem 9.2]{BT}), by  \cite[Theorem 9.4]{BT} there exists $\psi\in \IRR(B_0(X_i))$ lying over $\eta$. By \cite[Corollary 11.29]{CTFG}, $\psi(1)_p = \eta(1)_p = p$ and, by the Clifford correspondence, $\chi =\psi^G\in \IRR(G)$. By \cite[Corollary 6.2]{BT} and by Brauer's third main theorem (\cite[Theorem~6.7]{BT}), $\chi\in B_0(G)$. Since  $P\subset X_i$,  it follows that
\[\chi(1)_p = |G:X_i|_p \eta(1)_p = p,\]
and Theorem~\ref{thmA:A} holds in this case. Thus we may assume that $|X_i:S_iC_i|$ is a $p'$-number for all $1\leqslant i\leqslant t$.

Let $N\lhd G$ be the kernel of the action of $G$ on $\{S_i\}_{i=1}^t$ by conjugation. Since $P$ acts trivially on this set,  $P\subset N$ and $|G:N|$ is a $p'$-number. Now, 
\[N/\cn[N](S_i) = N/(N\cap C_i)\cong NC_i/C_i\subset X_i/C_i\]
and, writing $C = \bigcap_{i=1}^t C_i =\cn(K)\subset N$, we obtain a natural monomorphism
\[N/C \rightarrow  \prod_{i=1}^t X_i/C_i\, .\]
The image of $KC/C$ under this monomorphism is $\prod_{i=1}^t S_iC_i/C_i$. Thus $N/KC\cong (N/C)/(KC/C)$ is isomorphic to a subgroup of 
\[\prod_{i=1}^t (X_i/C_i)/(S_iC_i/C_i)\cong \prod_{i=1}^t X_i/S_iC_i\, .\]
Since $|X_i:S_iC_i| $ is not divisible by $p$ for all $1\leqslant i\leqslant t$, we conclude that $p$ does not divide $|N:KC|$. Since $p$ does not divide $|G:N|$, the result follows. 
\end{proof}

We are finally ready to prove Theorem~\ref{thmA:A}.

\begin{proof}[Proof of Theorem~\ref{thmA:A}]
We proceed by induction on $|G|$. Let $P \in {\rm Syl}_p(G)$ be nonabelian.
By \cite[Theorem 6.10]{BT}, we may assume that $\op_{p'}(G)=1$.
Suppose that $N \lhd G$ is proper and has $p'$-index. By the inductive hypothesis,
there exists $\gamma \in {\rm Irr}(N)$ in the principal block such that $1<\gamma(1)_p \le m(P)$.
Now by \cite[Theorems~9.2,~9.4]{BT}, there is $\chi \in {\rm Irr}(G)$ in the principal block over $\gamma$.
By \cite[Corollary~11.29]{CTFG}, we have that $\chi(1)_p=\gamma(1)_p$, and we are done in this case.
Hence ${\bf O}^{p'}(G)=G$.

Let $K$ be a minimal normal subgroup of $G$. Suppose first that $K$ is abelian.
Then $K$ is a $p$-group.
By Theorem~\ref{thm:remove_pabelian}, we may assume that $m(P/K)=m(P)$.  
In particular, $P/K$ is nonabelian. By the inductive hypothesis, there exists  $\chi\in \IRR(B_0(G/K))\subset \IRR(B_0(G)) $ with $1\ls \chi(1)_p  \leqslant m(P/K)=m(P)$ and we are done in this case.

Thus we may assume that $K$ is nonabelian and, by Theorem~\ref{thm:nonabelian_min}, we may assume that $K\times \cn(K)$ is of $p'$-index. Therefore
$G=K\times C$, where $C=\cn(K)$. Since $P=(P\cap K)\times (P\cap C)$, we have $m(P)=\min\{ m(P\cap K),m(P\cap C)\}$. If $m(P)=m(P\cap C)$, the inductive hypothesis applies and we may let $\chi\in \IRR(B_0(C))$ with $1\ls \chi(1)_p \leqslant m(P\cap C)=m(P)$. Then $1_K\times \chi\in \IRR(B_0(G))$ does the job. If $m(P)=m(P\cap K)$ and $K\ls G$, we can similarly prove the result. Thus we may assume that $G=K$. Since $K$ is minimal normal, it is a nonabelian simple group. The result now follows by the main result of \cite{BM}. 
\end{proof}

\end{document}